\documentclass[12pt, reqno]{amsart}
\usepackage[normalem]{ulem}

\usepackage{amsmath, amsthm, amssymb}
\usepackage{enumerate}
\usepackage[margin=1.0in]{geometry}
\usepackage{xcolor}
\definecolor{cite}{rgb}{0.30,0.60,1.00}
\definecolor{url}{rgb}{0.00,0.00,0.80}
\definecolor{link}{rgb}{0.40,0.10,0.20}
\usepackage[pdfusetitle,colorlinks,linkcolor=link,urlcolor=url,citecolor=cite,pagebackref,breaklinks]{hyperref}
\usepackage{graphicx}
\usepackage{cleveref}
\usepackage{mathdots}
\usepackage{tikz-cd}
\usepackage{comment}
\usepackage[all]{xy}
\usepackage{mathtools}
\usepackage{float}

\newtheorem{theorem}{Theorem}[section]

\newtheorem{proposition}[theorem]{Proposition}
\newtheorem{lemma}[theorem]{Lemma}
\newtheorem{conjecture}[theorem]{Conjecture}
\newtheorem{corollary}[theorem]{Corollary}

\theoremstyle{definition}

\theoremstyle{definition}
\newtheorem{remark}[theorem]{Remark}
\theoremstyle{definition}

\newcommand{\F}{\mathbb{F}}

\newcommand{\Z}{\mathbb{Z}}
\newcommand{\Gr}{\mathrm{Gr}}

\newcommand{\PGL}{\mathrm{PGL}}
\newcommand{\GL}{\mathrm{GL}}
\newcommand{\Gal}{\mathrm{Gal}}
\newcommand{\Fr}{\mathrm{Fr}}
\newcommand{\Tr}{\mathrm{tr}}
\newcommand{\Aut}{\mathrm{Aut}}

\newcommand{\Spec}{\operatorname{Spec}}
\newcommand{\OPtd}{\Pi_{k,n}^\circ}
\newcommand{\OPtdtwist}{\Pi_{k,n}^{\circ'}}

\newcommand{\qbinom}[2]{\begin{bmatrix} #1 \\ #2 \end{bmatrix}_q}

\hypersetup{pdfauthor={Calvin Yost-Wolff,},
	pdfsubject={Combinatorics, Algebraic Geometry, Representation theory},
	pdfkeywords={Singer Cycle, Positroids, Grassmanians}}

\title[Singer Cycle Problem]{Point Count of the Top-dimensional Open Positroid Variety}

\author{Calvin Yost-Wolff}
\address{Department of Mathematics, University of Michigan, 3084 East Hall, 530 Church Street, Ann Arbor, MI 48109-1043 USA}
\email{calvinyw@umich.edu}

\makeatletter
\let\orig@lbibitem\@lbibitem
\def\@lbibitem[#1]#2{%
  \def\@tempkey{#2}%
  \def\@moorekey{Moore1896}%
  \ifx\@tempkey\@moorekey
    \orig@lbibitem[Moo1896]{#2}%
  \else
    \orig@lbibitem[#1]{#2}%
  \fi
}
\makeatother

\begin{document}

\begin{abstract}
In \cite{GLqtcat}, Galashin and Lam discovered that when $k$ and $n$ are coprime, the proportion of subspaces in $\mathrm{Gr}(k,n)(\mathbb{F}_q)$ that lie in the top-dimensional open positroid variety $\Pi_{k,n}^\circ(\mathbb{F}_q)$ is
$|(\mathbb{F}_q^\times)^n|/|\mathbb{F}_{q^n}^\times|$. In this paper, I recover this point count identity by relating the split torus action on $(\Pi_{k,n}^\circ)_{\mathbb{F}_q}$ and an anisotropic torus action on a $\mathbb{F}_q$ rational form of $\Pi_{k,n}^\circ$. 
The main step in the point count argument and the main technical result in this paper is that cyclic rotation acts trivially on the torus-equivariant cohomology of $\Pi_{k,n}^\circ$ when $k$ and $n$ are coprime.
\end{abstract}
\maketitle

\tableofcontents

\section{Introduction}
Building on Postnikov's study of the totally nonnegative Grassmannian and its cell
decomposition~\cite{Pos06}, Knutson--Lam--Speyer constructed a stratification of $\Gr(k,n)$
by \emph{open positroid varieties} $\Pi_f^\circ$ indexed by bounded affine permutations
$f\in B_{k,n}$~\cite{KLS13}. Among all open positroid strata there is a unique one of maximal (i.e. $\dim \Gr(k,n)$)
dimension.  Let $f_{k,n}$ be the bounded affine permutation defined by
$f_{k,n}(i)=i+k$ on $\mathbb{Z}$.  The corresponding stratum $\OPtd \;:=\; \Pi_{f_{k,n}}^\circ$
is the \emph{top-dimensional open positroid variety}.
In Pl\"ucker coordinates, it admits a concrete description:
if $I_r=\{r+1,r+2,\dots,r+k\}$ denotes the cyclic interval (indices taken modulo $n$), then
\begin{equation*}
\OPtd
=\Bigl\{V\in \Gr(k,n)\ \Big|\ \Delta_{I_r}(V)\neq 0\text{ for all }r\in\mathbb{Z}/n\mathbb{Z}\Bigr\}.
\end{equation*}
i.e. $\OPtd$ is the complement in $\Gr(k,n)$ of the union of the vanishing locus of the $n$
cyclic Pl\"ucker coordinates $\{\Delta_{I_r}=0\}$.

\subsection{Finite-field point counts and \texorpdfstring{$R$}{R}-polynomials}

In \cite{GLqtcat}, Galashin--Lam note that open Richardson varieties $\F_q$-point counts are 
given by evaluations of $R$-polynomials.  The top-dimensional open positroid varieties $\OPtd$ are isomorphic projections of Richardson varieties 
whose point counts Galashin--Lam find the following nice formula for when $k,n$ are coprime: Set
\[
[n]_q:=1+q+\cdots+q^{n-1},
\qquad
\qbinom{n}{k}:=\frac{[n]_q!}{[k]_q!\,[n-k]_q!},
\]
then
\begin{align*}
\#\OPtd(\F_q)
=(q-1)^{\,n-1}\cdot \frac{1}{[n]_q}\qbinom{n}{k},
\qquad\text{when }\gcd(k,n)=1.
\end{align*}
The normalized count $\frac{1}{(q-1)^{n-1}}\#\OPtd(\F_q)$ is the
classical ``rational $q$-Catalan'' polynomial $\frac{1}{[n]_q}\qbinom{n}{k}$ in the sense of
MacMahon~\cite{MacMahon1960}. Notice that the rational $q$-Catalan polynomial is equal to $|\Gr(k,n)(\mathbb{F}_q)|/|(\mathbb{F}_{q^n}^\times/\mathbb{F}_q^\times)|$.

Classically, the Grassmannian point count
\[
\#\Gr(k,n)(\F_q)=\qbinom{n}{k}
\]
is given by the Gaussian binomial coefficient.  Geometrically, this is explained by the
Schubert cell decomposition or by the purity of the
cohomology of $\Gr(k,n)$. Thus when $k$ and $n$ are coprime
\begin{equation}\label{eq:point-count-top}
    \#\OPtd(\F_q)
    =\frac{(q-1)^{\,n-1}}{[n]_q}|\Gr(k,n)(\F_q)| = \frac{|(\F_q^\times)^n|}{|\F_{q^n}^\times|}|\Gr(k,n)(\F_q)|
\end{equation}

Thomas Lam and Victor Reiner posed the following question:

\begin{quote}
\emph{Can we find an explanation of the formula \eqref{eq:point-count-top} in terms of the action of an anisotropic torus on the Grassmannian?}
\end{quote}

In this paper, I develop the perspective that the nice formula \eqref{eq:point-count-top}
is the result of the cyclic rotation automorphism acting trivially on the $T$-equivariant compactly supported \'etale cohomology of $\OPtd$ for $k,n$ coprime.

Let us rephrase \eqref{eq:point-count-top} slightly: Let $T\cong(\mathbb{G}_m)^{n}/\mathbb{G}_m$ be the diagonal torus in $\PGL_n$ acting on $\Gr(k,n)$.
When $\gcd(k,n)=1$ the $T$-action on $\OPtd$ is free;
the quotient
\[
X_{k,n}^\circ \;:=\;\OPtd/T = \Spec \Bbbk[\OPtd]^T
\]
is then a smooth affine variety called a \emph{positive configuration space}\cite{GLqtcat}. Let $T'$ be an anisotropic torus in $\PGL(n)_{\F_q}$, so $T'(\mathbb{F}_q) = \mathbb{F}_{q^n}^\times/\mathbb{F}_q^\times$. Choose $g \in \PGL(n)(\mathbb{F}_{q^n})$ such that $gTg^{-1} = T'$ and the Lang map applied to $g$ (i.e $g^{-1}\Fr_q(g)$ where $\Fr_q$ is the geometric Frobenius) is the image of the cyclic rotation matrix $\rho$ in $\PGL(n)(\mathbb{F}_{q})$. Define the twisted top-dimensional open positroid variety 
\begin{equation*}
\OPtdtwist
\;:=\;\Bigl\{V\in \Gr(k,n)\ \Big|\ \Delta_{I_r}(g^{-1}V)\neq 0\text{ for all }r\in\mathbb{Z}/n\mathbb{Z}\Bigr\}.
\end{equation*}
The condition $g^{-1}\Fr(g) = \rho$, the $\rho$ stability of $\OPtd$, and the fact that $\rho$ normalizes $T$ combine to imply that $X_{k,n}^{\circ'}:= \OPtdtwist/T'$ has an $\mathbb{F}_q$-rational structure.
Although $X_{k,n}^{\circ}$ and $X_{k,n}^{\circ'}$ are not necessarily isomorphic over $\mathbb{F}_{q}$ (see remark \ref{rem:bij_hope}), I will show their compactly supported \'etale cohomologies are isomorphic as $\Gal(\overline{\mathbb{F}_q}/\mathbb{F}_q)$-modules when $k$ and $n$ are coprime, implying that their $\mathbb{F}_q$-point counts are the same. This will follow from the fact that $\rho$ acts trivially on the compactly supported \'etale cohomology of $X_{k,n}^\circ$ when $k$ and $n$ are coprime.

\subsection{Main Results}
Let $\Bbbk$ be any field, let $V = \Bbbk^n$ with standard basis $e_1,\dots,e_{n}$. The rotation automorphism
\[
\rho : V \to V, \qquad \rho(e_i) = e_{i+1} \quad (1 \le i \le n-1), \qquad \rho(e_{n}) = e_1,
\]
is the linear operator given by cyclic rotation of the coordinates.
$\rho$ acts on the Grassmannian $\Gr(k,V) \cong \Gr(k,n)$ by $\rho\cdot W := \rho(W)$. This action preserves $\OPtd \subset \Gr(k,V)$ and normalizes the $T$-action on $\Gr(k,V)$ via 
\[
\rho\cdot t\cdot W = \rho(t) \cdot \rho \cdot W
\]
where 
\[
    \rho(\operatorname{diag}(t_1,t_2, \dots, t_n)) = \operatorname{diag}(t_n,t_1,t_2, \dots, t_{n-1}).
\]
Thus $\rho$ acts $H_{T}^*(\mathrm{pt})$-semilinearly on $H_{T,c}^*(\OPtd) = H_c^*(X_{k,n}^\circ)$. The following is the main theorem in this paper

\begin{theorem}\label{thm:main_thm_1.1}
    Let $k$ and $n$ be coprime, and let $\Bbbk$ be an algebraically closed field with characteristic coprime to $\ell$ and $n$. Then $\rho$ acts trivially on $H_c^*\big((X_{k,n}^\circ)_{\Bbbk},\mathbb{Q}_{\ell}\big)$.
\end{theorem}
The above theorem implies that the natural isomorphism of the compactly supported \'etale cohomologies of $X_{k,n}^\circ$ and $X_{k,n}^{\circ'}$ given by $m_g^*$ is in fact an isomorphism of $\Gal(\overline{\mathbb{F}_q}/\F_q)$-representations (i.e. as $\overline{\mathbb{Q}_{\ell}}[\Fr]$-modules).

I will then show the relationship in point counts:
\begin{theorem}\label{thm:main_thm_1.2}
    Let $k$ and $n$ be coprime and let $q$ be coprime to $n$. Then
    \begin{align*}
    \frac{|\OPtd(\mathbb{F}_q)|}{|T(\mathbb{F}_q)|} = \frac{|\OPtdtwist(\mathbb{F}_q)|}{|T'(\mathbb{F}_q)|} = \frac{|\Gr(k,n)(\mathbb{F}_q)|}{|T'(\mathbb{F}_q)|}.
    \end{align*}
\end{theorem}
with the last equality following from a Moore determinant argument showing that $|\OPtdtwist(\mathbb{F}_q)| = |\Gr(k,n)(\mathbb{F}_q)|$.

I will cover the necessary preliminaries in section~\ref{sec:prelims} before proving the main theorems in section~\ref{sec:main}.
\subsection{Acknowledgments}
Thank you to Thomas Lam for suggesting this problem to me. Thank you to Thomas Lam and David Speyer for insightful discussions related to this project. The author of this paper was supported by NSF RTG grant DMS 1840234 while working on this project.
\section{Preliminaries}\label{sec:prelims}
An important well-known property of $\OPtd$ when $k,n$ are coprime is that the natural torus action of the split torus $T = \mathbb{G}_m^n/\mathbb{G}_m$ in $\PGL(n)$ is free and the quotient $\OPtd/T$ is smooth. \cite[Proposition~1.6]{GLqtcat} proves a more general statement about when $T$ acts freely on $\Pi^\circ_f$. I include a simpler argument for $\OPtd$.

\begin{lemma}
When $k$ and $n$ are coprime, the $T$-action on $\OPtd$ is free.
\end{lemma}

\begin{proof}
Fix $W \in \OPtd(\Bbbk) \subseteq \Gr(k,n)(\Bbbk)$ and suppose $t\in T(\Bbbk)$ stabilizes $W$. Choose a lift
$\widetilde t=(t_1,\dots,t_n)\in (\Bbbk^\times)^n$ and a $k$-by-$n$ matrix $M_W$ whose rows form a basis of $W$. In terms of the Pl\"ucker embedding, the condition
$t\cdot W=W$ implies there exists $\lambda \in \Bbbk^\times$ such that for any $k$-subset $I\subset [n]$,
\[
    \Delta_I\left(M_W \cdot \widetilde t\right) \;=\; \lambda\Delta_I(M_W)
\]
where $\Delta_I$ is the Pl\"ucker coordinate for the subset $I$. The diagonal action of $\widetilde{t}$ scales the Pl\"ucker coordinate by the
character $\prod_{i\in I} t_i$, i.e.
\[
    \Delta_I\left(M_W \cdot \widetilde t\right) \;=\; \Bigl(\prod_{i\in I} t_i\Bigr)\,\Delta_I(M_W).
\]
Since $W \in \OPtd(\Bbbk)$, $\Delta_{[i,i+k-1]}(M_W)\neq 0$ for all cyclic intervals $[i,i+k-1]$. Comparing coefficients in
$\Delta_{[i,i+k-1]}(M_W \cdot \widetilde t) = \lambda\Delta_{[i,i+k-1]}(M_W)$ gives, for every $i\in \Z/n\Z$,
\[
t_i\,t_{i+1}\cdots t_{i+k-1} \;=\; \lambda \quad\text{(indices mod $n$)}.
\]
Dividing the equation for $i+1 \bmod n$ by the equation for $i$ yields
\[
\frac{t_{i+1}\cdots t_{i+k}}{t_i\cdots t_{i+k-1}} \;=\; 1
\qquad\Longrightarrow\qquad
t_{i+k}\;=\;t_i
\quad\text{(indices mod $n$)}.
\]
Thus $t_{i+mk}=t_i$ for all $m$. Because $\gcd(k,n)=1$, the map $i\mapsto i+k$ generates all
residues modulo $n$, so all coordinates are equal: $t_1=\cdots=t_n=a$ and hence $t = 1$ in $\PGL(n)(\Bbbk)$.
\end{proof}

\begin{lemma}\label{lem:smoothness}
    The scheme theoretic quotient $(\OPtd/T)_{\Bbbk} = (X_{k,n}^\circ)_{\Bbbk}$ is smooth over $\Bbbk$.
\end{lemma}

I will prove this Lemma at the end of subsection~\ref{subsec:quotients_H1} after a necessary discussion of quotients.

\subsection{Cohomology of the top-dimensional open positroid variety}\label{subsec:cohom}
Galashin--Lam determine the mixed Hodge polynomial of
$P(\Pi_f^\circ;q,t)$ in terms of higher Ext groups of Verma modules or homology of Rouquier complexes~\cite{GLqtcat}. This mixed Hodge polynomial recovers both the ordinary Poincar\'e polynomial
and the finite-field point counts (after a standard renormalization). Galashin--Lam's Theorem 1.1 computes the Poincar\'e polynomial of $\OPtd$ and implies the weaker statement I will need:

\begin{theorem}
\label{thm:cohom-dim}
When $k$ and $n$ are coprime, the cohomology of $H_c^*(X^\circ_{k,n})$ is concentrated in even degrees with Euler characteristic given by the rational Catalan number
\[
    \chi_c\bigl(X^\circ_{k,n}\bigr) = \mathrm{dim}\bigl(H_c^*(X^\circ_{k,n})\bigr) = \frac{1}{n}\binom{n}{k}.
\]
\end{theorem}

The original motivation of Galashin and Lam comes from a conjecture arising from the works of \cite{STZ17} and \cite{STWZ19}:

\begin{conjecture}[\cite{GLqtcat}, Conjecture~1.21]
\leavevmode\par
    \begin{enumerate}
    \item There exists a non-algebraic deformation retract 
    \[
    r:X^\circ_{k,n} \to J_{k,n-k}
    \]
    from $X^\circ_{k,n}$ onto the compactified Jacobian $J_{k,n-k}$ of the plane curve singularity $x^k=y^{n-k}$.
    \item In compactly supported cohomology,
    \[
    r^*: H_c^*(J_{k,n-k}) \to H_c^*(X^\circ_{k,n})
    \] 
    sends the weight filtration to the perverse filtration.
    \end{enumerate}
\end{conjecture}

Such an $r$ is expected to arise from a wild non-abelian Hodge correspondence. Constructing such a retraction $r$---without necessarily satisfying (2)---would still imply Theorem~\ref{thm:cohom-dim}. Adding this piece to the argument in this paper would yield a simpler proof of formula \eqref{eq:point-count-top} without passing through the work of Galashin--Lam, which already proves a stronger statement than formula \eqref{eq:point-count-top}. Vivek Shende has an explicit idea on how to construct such a deformation $r$ \cite{ShendeBlog}, although to my knowledge there are still several steps to be completed.


\begin{remark}
    Constructing a retract $r$ as above which also is $\rho$-equivariant could lead to another proof of Theorem~\ref{thm:main_thm_1.1} for characteristic zero fields: such a retraction would deform the rotation automorphism $\rho$ to some other possibly non-algebraic automorphism $\overline{\rho}$ of the compactified Jacobian $J_{a,b}$. This $\overline{\rho}$ would likely be homotopic to the identity and so it would act trivially on the compactified Jacobian's cohomology. This would imply Theorem~\ref{thm:main_thm_1.1} over $\mathbb{C}$ and applying smooth base change to the compactly supported \'etale cohomology of $H_c^*(X^\circ_{k,n})$ would give the result over characteristic zero fields.
\end{remark}

\subsection{Grothendieck--Lefschetz}\label{subsec:GL-rationality}

Throughout this subsection let $q=p^r$ be a power of the prime $p$, let $\ell\neq p$ be a fixed prime,
and let $X$ be a variety over $\mathbb{F}_q$. Let
$\Fr_q\in\Gal(\overline{\mathbb{F}}_q/\mathbb{F}_q)$ denote the \emph{geometric Frobenius},
i.e.\ the inverse of the arithmetic Frobenius $a\mapsto a^q$.
For any $i\ge 0$, the $\ell$-adic compactly supported \'etale cohomology
$H^i_{c}(X_{\overline{\mathbb{F}}_q},\mathbb{Q}_\ell)$ carries a continuous action of
$\Gal(\overline{\mathbb{F}}_q/\mathbb{F}_q)$, and hence an endomorphism induced by $\Fr_q$.

\begin{theorem}[\cite{SGA4half}, Grothendieck--Lefschetz trace formula]\label{thm:GL-trace}
Let $X$ be a separated scheme of finite type over $\mathbb{F}_q$ and let $\ell\neq p$.
Then
\[
\#X(\mathbb{F}_q)
=
\sum_{i\ge 0}(-1)^i\Tr\!\Bigl(\Fr_q \,\big|\,
H^i_{c}(X_{\overline{\mathbb{F}}_q},\mathbb{Q}_\ell)\Bigr).
\]
More generally, for every $m\ge 1$,
\[
\#X(\mathbb{F}_{q^m})
=
\sum_{i\ge 0}(-1)^i\Tr\!\Bigl(\Fr_q^{\,m} \,\big|\,
H^i_{c}(X_{\overline{\mathbb{F}}_q},\mathbb{Q}_\ell)\Bigr).
\]
\end{theorem}

This theorem mirrors the classical Lefschetz fixed point theorem, which expresses
the number of fixed points of an automorphism as an alternating sum of traces on cohomology.
In general, the classical Lefschetz theorem holds only for proper varieties.
However, for finite order automorphisms, an analogous statement remains valid for
quasi-projective varieties:

\begin{theorem}[\cite{DeligneLusztig1976}, Theorem~3.2 Lefschetz for finite order automorphisms]\label{thm:lef-finiteorder}
Let $X$ be a quasi-projective variety over an algebraically closed field $\Bbbk$,
let $\ell\neq \operatorname{char}(\Bbbk)$, and let $\sigma$ be a finite order automorphism
whose order is relatively prime to $\operatorname{char}(\Bbbk)$.
Assume that there are finitely many fixed points of $X^\sigma$ and $X$ is smooth.
Then
\[
\#X^\sigma
=
\sum_{i\ge 0}(-1)^i\Tr\!\Bigl(\sigma \,\big|\,
H^i_{c}(X,\mathbb{Q}_\ell)\Bigr).
\]
\end{theorem}
It is important for the above formula that $X$ is smooth since this implies the fixed point set of $\sigma$ is smooth \cite[Proposition~3.4]{EDIXHOVEN}. In this case when $X^\sigma$ is finitely many fixed points, the fixed points of $X^\sigma$ are reduced isolated points.
\subsection{Rational forms and nonabelian \texorpdfstring{$H^1$}{H1}.}\label{subsec:rat_forms}
Let $\Bbbk'/\Bbbk$ be a finite Galois extension with group $G=\Gal(\Bbbk'/\Bbbk)$.
Let $Y'$ be a $\Bbbk'$-variety.
A \emph{$\Bbbk$-form} (or \emph{$\Bbbk$-rational structure}) on $Y'$ is a $\Bbbk$-variety $Y$ together with
a $\Bbbk'$-isomorphism $Y_{\Bbbk'}\simeq Y'$.
Two forms are identified if there is a $\Bbbk$-isomorphism between the $\Bbbk$-models compatible with the
chosen $\Bbbk'$-isomorphisms.

In general, descent data for $Y'$ relative to $\Bbbk'/\Bbbk$ are encoded by \emph{nonabelian} Galois cohomology. 
A nonabelian $1$-cocycle is a map $c\colon G\to \Aut_{\Bbbk'}(Y')$ satisfying
\[
c(gh) = c(g) \cdot {}^{g}c(h) \qquad (g,h\in G),
\]
and two cocycles $c,c'$ are cohomologous if there exists $\phi\in \Aut_{\Bbbk'}(Y')$ with
\[
c'(g) = \phi^{-1}\, c(g)\, {}^{g}\phi \qquad (g\in G).
\]
The set of equivalence classes of nonabelian $1$-cocycles is the pointed set $H^1(G,\Aut_{\Bbbk'}(Y'))$.

\begin{proposition}[Forms classified by \texorpdfstring{$H^1$}{H1}]\label{prop:forms-H1}
(See \cite[\S I.5]{SerreGC}, \cite[Prop.~4.4.4, Thm.~4.5.2]{Poonen17}.)
Let $\Bbbk'/\Bbbk$ be a finite Galois extension with group $G$ and let $Y'$ be a $\Bbbk'$-variety.
Isomorphism classes of $\Bbbk$-forms of $Y'$ are classified by the nonabelian cohomology set
$H^1(G,\Aut_{\Bbbk'}(Y'))$. Additionally, if an open set $U' \subset Y'$ is stable under the cocycle isomorphisms, 
then $U'$ descends to an open subset $U \subseteq Y$ with $U_{\Bbbk'} = U'$.
\end{proposition}

I now return to my primary focus: point-counting over finite fields.
Let $X$ be a variety defined over $\mathbb{F}_q$ and fix $n\ge 1$.
The extension $\mathbb{F}_{q^n}/\mathbb{F}_q$ is Galois with cyclic group
\[
G=\Gal(\mathbb{F}_{q^n}/\mathbb{F}_q)\cong \mathbb{Z}/n\mathbb{Z},
\qquad
G=\langle \Fr_q\rangle,
\]
where $\Fr_q$ denotes the \emph{geometric Frobenius} on scalars, i.e.\ the inverse of the $q$-power map.
Equivalently, its restriction to $\mathbb{F}_{q^n}$ is
\[
\Fr_q(a)=a^{q^{\,n-1}}\qquad (a\in \mathbb{F}_{q^n}),
\]
so that $\Fr_q^{-1}(a)=a^q$ is the arithmetic Frobenius.

In the cyclic case, a $1$-cocycle $c\colon G\to \Aut_{\mathbb{F}_{q^n}}(X_{\mathbb{F}_{q^n}})$ is determined
by the image of $\Fr_q$
\[
\alpha:=c(\Fr_q)\in \Aut_{\mathbb{F}_{q^n}}(X_{\mathbb{F}_{q^n}}),
\]
and the cocycle condition is equivalent to the relation
\begin{equation}\label{eq:cocycle-cyclic}
1=c(\Fr_q^n)
= \alpha\cdot \Fr_q(\alpha)\cdot \Fr_q^2(\alpha)\cdots \Fr_q^{n-1}(\alpha)
\qquad \text{in }\Aut_{\mathbb{F}_{q^n}}(X_{\mathbb{F}_{q^n}}).
\end{equation}
Moreover, changing the cocycle by a coboundary amounts to \emph{Frobenius-twisted conjugacy}:
\begin{equation*}
\alpha \sim \alpha'
\quad\Longleftrightarrow\quad
\exists\, a\in \Aut_{\mathbb{F}_{q^n}}(X_{\F_{q^n}})\ \text{such that}\ 
\alpha' = a^{-1}\, \alpha\, \Fr_q(a).
\end{equation*}
Thus the $\mathbb{F}_q$-forms of $X_{\F_{q^n}}$ are in bijection with $\Fr_q$-twisted conjugacy classes
of elements $\alpha$ satisfying \eqref{eq:cocycle-cyclic}.

Fix such an element $\alpha\in \Aut_{\mathbb{F}_{q^n}}(X_{\F_{q^n}})$ representing a twisted conjugacy class.
Let $X_\alpha$ denote the corresponding $\mathbb{F}_q$-form of $X_{\mathbb{F}_{q^n}}$ (the ``twist by $\alpha$'').
By construction there is an isomorphism over $\mathbb{F}_{q^n}$,
\[
\iota_\alpha : (X_\alpha)_{\mathbb{F}_{q^n}} \xrightarrow{\ \sim\ } X_{\mathbb{F}_{q^n}},
\]
which then identifies their geometric compactly supported \'etale cohomology.

\begin{proposition}[Geometric Frobenius on a twisted form]\label{prop:Frob-twist}
Let $X_\alpha$ be the $\mathbb{F}_q$-form of $X_{\mathbb{F}_{q^n}}$ determined by the class of
$\alpha\in \Aut_{\mathbb{F}_{q^n}}(X_{\F_{q^n}})$.
Under the identification
\[
H^i_{c}\bigl((X_\alpha)_{\overline{\mathbb{F}}_q},\mathbb{Q}_\ell\bigr)
\ \xrightarrow[\ \iota_\alpha^*\ ]{\ \sim\ }\
H^i_{c}\bigl(X_{\overline{\mathbb{F}}_q},\mathbb{Q}_\ell\bigr),
\]
the action of geometric Frobenius satisfies
\[
\Fr_q^{(X_\alpha)}
=
\alpha^* \circ \Fr_q^{(X)},
\]
where $\alpha^*$ denotes the pullback action of the automorphism $\alpha$ on cohomology.
\end{proposition}

\begin{proof}
The cocycle defining $X_\alpha$ twists the $\Gal(\overline{\mathbb{F}}_q/\mathbb{F}_q)$-action on
$(X_\alpha)_{\overline{\mathbb{F}}_q}$ so that, under $\iota_\alpha$, the action of $\Fr_q$ is modified by $\alpha$.
Passing to compactly supported $\ell$-adic cohomology gives
$\Fr_q^{(X_\alpha)}=\alpha^*\circ \Fr_q^{(X)}$.
\end{proof}

\subsection{Quotients, torsors and \texorpdfstring{$H^1$}{H1}-obstruction}\label{subsec:quotients_H1}
Let $\Bbbk$ be a field, $G$ be a geometrically reductive smooth affine group scheme over $\Bbbk$, and $X$ be a $\Bbbk$-scheme
equipped with a right $G$-action $a\colon X\times_\Bbbk G\to X$.

Let $(X/G)_{\mathrm{fpqc}}$ denote the fpqc sheafification of the presheaf $T \longmapsto X(T)/G(T)$.
In the affine case (i.e. $X=\Spec A$)
the invariant ring $A^G$ is finitely generated and the morphism
\[
\Spec A \longrightarrow \Spec(A^G)
\]
is the scheme-theoretic quotient \cite{MFK,GIT-Haboush}. In the case that $G$ acts freely, this agrees with the fpqc quotient $(X/G)_{\mathrm{fpqc}}$
\cite[Prop.~9.7.8]{Alper14}.
For the rest of this section I will assume a scheme $Y$ represents the quotient $(X/G)_{fpqc}$.

Assume also that the $G$-action is free.
Then the quotient map $\pi\colon X\to Y$ is a right
$G$-torsor. Fix an algebraic closure $\bar{\Bbbk}$ and write $\Gamma:=\Gal(\bar{\Bbbk}/\Bbbk)$.
Let $y\in Y(\Bbbk)$ and choose a lift
\[
x \in X_y(\bar{\Bbbk})\subset X(\bar{\Bbbk})
\qquad\text{with}\qquad \pi(x)=y.
\]
Because the $G$-action is free and transitive on geometric fibres, for each $\gamma\in\Gamma$ there exists a
\emph{unique} $c_\gamma\in G(\bar{\Bbbk})$ such that
$\gamma(x)=x\cdot c_\gamma.$
Then $c=(c_\gamma)_{\gamma\in\Gamma}$ is a $1$-cocycle:
for $\gamma,\delta\in\Gamma$
\[
(\gamma\delta)(x)
= \gamma(\delta(x))
= \gamma(x\cdot c_\delta)
= \gamma(x)\cdot \gamma(c_\delta)
= (x\cdot c_\gamma)\cdot \gamma(c_\delta)
= x\cdot\bigl(c_\gamma\,\gamma(c_\delta)\bigr),
\]
hence by uniqueness $c_{\gamma\delta}=c_\gamma\,\gamma(c_\delta)$.
Changing the lift $x$ to $x':=x\cdot g$ with $g\in G(\bar{\Bbbk})$ replaces $c$ by the cohomologous cocycle
$c'_\gamma=g^{-1}c_\gamma\gamma(g)$, so the class $[c]\in H^1(\Gamma,G)$ depends only on $y$.
Thus I obtain a well-defined map of pointed sets
\[
\delta\colon Y(\Bbbk)\longrightarrow H^1(\Gamma,G),
\qquad
y\longmapsto [X_y].
\]

The class $\delta(y)$ is precisely the obstruction to lifting $y$ to a $\Bbbk$-point of $X$:
indeed $\delta(y)=1$ if and only if $X_y$ is the trivial $G$-torsor which happens if and only if $X_y(\Bbbk)\neq\varnothing$ \cite[Proposition~5.12.14]{Poonen17}.

\begin{lemma}[obstruction to lifting]\label{lem:H1_lift_obstruction}
Assume $G$ acts freely on $X$ and that the fpqc quotient is represented by a $\Bbbk$-scheme $Y$.
Then there is an exact sequence of pointed sets
\[
1 \longrightarrow G(\Bbbk) \longrightarrow X(\Bbbk) \xrightarrow{\ \pi\ } Y(\Bbbk)
\xrightarrow{\ \delta\ } H^1(\Gamma,G).
\]
\end{lemma}

In this paper, $X=\OPtd$ or $X=\OPtdtwist$. Both are affine, and our $G$ will be a connected torus acting freely on $X$ over $\F_q$. Over a finite field, the $H^1$ obstruction vanishes for connected groups by Lang's theorem:

\begin{theorem}[Lang's theorem\cite{Lang1956,Steinberg1968}]\label{thm:LangS}
For any finite field $\F_q$ and smooth connected algebraic group $G$, $H^1(\Gal(\overline{\F}_q/\F_q),G)$ is a single point.
\end{theorem}

\begin{corollary}\label{cor:cor_of_LS}
Let $G$ be a connected smooth algebraic group over $\mathbb{F}_q$ and $X$ a variety with a free $G$-action.
Assume the fpqc sheaf quotient $X/G$ is represented by a $\mathbb{F}_q$-variety $Y$.
Then
\[
Y(\mathbb{F}_q)=X(\mathbb{F}_q)/G(\mathbb{F}_q).
\]
\end{corollary}

\begin{proof}
By Lemma~\ref{lem:H1_lift_obstruction}, there is an exact sequence
\[
1 \to G(\F_q) \to X(\F_q) \to Y(\F_q) \xrightarrow{\delta} H^1(\Gal(\overline{\F}_q/\F_q),G).
\]
By Lang's theorem, $H^1(\Gal(\overline{\F}_q/\F_q),G)=1$ for connected $G$,
so $\delta$ is trivial and $Y(\F_q)=X(\F_q)/G(\F_q)$.
\end{proof}

Before I begin the main arguments, I will make a brief aside showing the smoothness of $X_{k,n}^\circ$ which is important in later applications of the Lefschetz fixed point Theorem~\ref{thm:lef-finiteorder}:

\begin{proof}[Proof of Lemma~\ref{lem:smoothness}]
Let
\[
\pi:\OPtd \longrightarrow \OPtd/T = X_{k,n}^\circ
\]
denote the quotient morphism. Since $\OPtd$ is an open subscheme of a smooth variety, it is smooth.

Since $T$ acts freely on $\OPtd$, the quotient map $\pi$ is a principal $T$--bundle. Concretely, there exists an fpqc covering $\{U_i \to X_{k,n}^\circ\}$ such that
\[
\OPtd\times_{X_{k,n}^\circ} U_i \;\cong\; U_i \times T
\]
over $U_i$, and under this identification the base change of $\pi$ is just the projection
\[
U_i\times T \longrightarrow U_i.
\]
Since $T$ is smooth, this projection is smooth, hence $\pi$ is smooth and surjective.

Finally, since smoothness descends over smooth surjective morphisms $X_{k,n}^\circ=\OPtd/T$ is smooth.
\end{proof}

\section{Main Theorems}\label{sec:main}
Let $\rho\in\GL_n(\mathbb{Z})$ be the cyclic permutation (rotation) matrix which cyclically permutes a fixed standard coordinate basis of $\mathbb{Z}^n$. Let $T$ be the split torus in $\PGL(n)_{\mathbb{F}_q}$ and $T'$ be an anisotropic torus in $\PGL(n)_{\mathbb{F}_q}$. Then 
\[
    T(\mathbb{F}_q) = \frac{\left(\F_q^\times\right)^n}{\F_q^\times}, \qquad T'(\mathbb{F}_q) = \frac{\F_{q^n}^\times}{\F_q^\times}.
\]
Choose $g \in \PGL(n)(\mathbb{F}_{q^n})$ such that $gTg^{-1} = T'$ and $g^{-1}\Fr(g)$ is the image of the cyclic rotation matrix $\rho$ in $\PGL(n)(\mathbb{F}_{q})$. This is possible since the $\Fr$-twisted conjugacy class corresponding to the rational form $T'$ of $T=\mathbb{G}_m^n/\mathbb{G}_m$ is the conjugacy class of $n$-cycles in the Weyl group $S_n$ (Frobenius acts trivially on the Weyl group of $T$). 

Recall the twisted top-dimensional open positroid variety and twisted positive configuration space I defined as
\begin{equation*}
    \OPtdtwist
    \;:=\;\Bigl\{V\in \Gr(k,n)\ \Big|\ \Delta_{I_r}(g^{-1}V)\neq 0\text{ for all }r\in\mathbb{Z}/n\mathbb{Z}\Bigr\} \qquad X_{k,n}^{\circ'}:= \OPtdtwist/T'.
\end{equation*}

The image of $\rho$ in $\Aut(\OPtd)$ is the $\Fr$-twisted conjugacy class corresponding to the rational form $\OPtdtwist$ of the top-dimensional open positroid variety $\OPtd$. Similarly, the image of $\rho$ in $\Aut(X_{k,n}^\circ)$ is the $\Fr$-twisted conjugacy class corresponding to the rational form $X_{k,n}^{\circ'}$ of the positive configuration space $X_{k,n}^\circ$. 

In subsection~\ref{subsec:thm1.1}, I will show the main theorem that $\rho$ acts trivially in cohomology and then in subsection~\ref{subsec:main_cor} I will show how this implies the main corollary explaining the point count formula \eqref{eq:point-count-top}. Throughout the following sections I will use $X_{k,n}^\circ$ and $\OPtd/T$ interchangeably and I will use $X_{k,n}^{\circ'}$ and $\OPtdtwist/T'$ interchangeably.

\subsection{Proof of Theorem 1.1}\label{subsec:thm1.1}

The cyclic rotation fixed points in the Grassmannian were considered in \cite{Karp19}:

\begin{proposition}{\cite[Theorem~1.1]{Karp19}}\label{prop:rot_fixed_points}
Let $\Bbbk$ be a field containing a primitive $n$-th root of unity $\zeta$. Then the $\rho$-fixed points of $\Gr(k,n)(\Bbbk)$ are precisely the $ \binom{n}{k}$ subspaces obtained as spans of $k$-element subsets of the $n$ one-dimensional $\rho$-eigenspaces spanned by the vectors
\[
v_i:= (1,\zeta^i,\zeta^{2i},\dots,\zeta^{(n-1)i})^\top,\qquad i=0,\dots,n-1.
\]
\end{proposition}

\begin{proof}
$\rho$ is diagonalizable over $\Bbbk$ with $n$ distinct eigenvalues $\zeta^i$ ($0\le i<n$), and the eigenvectors may be taken as
\[
v_i=(1,\zeta^i,\zeta^{2i},\dots,\zeta^{(n-1)i})^\top .
\]
Note that $V=\bigoplus_{i=0}^{n-1} L_i$ where $L_i=\operatorname{span}\{v_i\}$.

Since each one dimensional subspace $L_i$ is $\rho$-stable, the $\binom{n}{k}$ subspaces spanned by size $k$ subsets of the $\{L_i\}$ are all $\rho$-stable. 

On the other hand, if $W\subset V$ is a $\rho$-stable $k$-dimensional subspace, then because $\rho$ acts semisimply on $V$, it acts semisimply on $W$. Thus $W$ decomposes as a direct sum of eigenspaces for $\rho$. So $W$ must be a direct sum of some of the eigenspaces $L_i$.
\end{proof}

\begin{corollary}\label{cor:rot_fixed}
    Let $\Bbbk$ be a field containing a primitive $n$-th root of unity $\zeta$. The $\rho$-fixed points of $\Gr(k,n)(\Bbbk)$ all lie in $\OPtd(\Bbbk)$.
\end{corollary}
\begin{proof}
Let $W$ be a $\rho$-fixed point of $\Gr(k,n)(\Bbbk)$. By the above Proposition, $W$ is represented by a $k$-by-$n$ matrix $(v_{i_1},v_{i_2}, \dots, v_{i_{k}})^\top$ where the $v_i$'s are as in the previous Proposition. Then the consecutive $k$-by-$k$ minor starting at column $r$ is the matrix whose $(a,b)$-entry is $\zeta^{i_a(r-1+b-1)}$. Rescaling each row by $\zeta^{i_a(r-1)}$ I get a Vandermonde matrix whose determinant is nonzero:
\begin{align*}
\det\left(\zeta^{i_a(r-1+b-1)}\right) &= \zeta^{(r-1)(\sum_{a=1}^{k} i_a)}\det\left(\zeta^{i_a(b-1)}\right) \\
&= \zeta^{(r-1)(\sum_{a=1}^{k} i_a)}\prod_{1 \leq a < b \leq k}\left(\zeta^{i_a}-\zeta^{i_b}\right) \neq 0.
\end{align*}
\end{proof}

After choosing a primitive $n$-th root of unity $\zeta$, define $V_I := \bigoplus_{i\in I} L_i$ for a subset $I \subset \{0,1,\dots,n-1\}$. Let $I+r := \{i+r \pmod{n} \mid i \in I\}$.

\begin{lemma}[Cyclic rotation fixed points in the quotient]\label{lem:rot_fixed_in_quotient}
    Let $\Bbbk$ be a field containing a primitive $n$-th root of unity $\zeta$. 
The $\rho$-fixed $\Bbbk$-points of the quotient $(\OPtd/T)(\Bbbk)$ are in natural bijection with cyclic rotation orbits of $k$-element subsets $I\subset\{0,\dots,n-1\}$. 
\end{lemma}

To prove the lemma I will need a proposition on how $T(\Bbbk)$ acts on the $\rho$-eigenvectors:
\begin{proposition}\label{prop:T_action_on_rot_fixed}
    Let $\Bbbk$ be a field containing a primitive $n$-th root of unity $\zeta$ and let $v_i$ be as in Proposition~\ref{prop:rot_fixed_points}.
Let $t \in T(\Bbbk)$, then there exist coefficients $c_m(t)\in \Bbbk$ such that
\[
    t v_i \;=\; \sum_{m=0}^{n-1} c_m(t)\, v_{i+m}
\]
where $i+m$ is interpreted modulo $n$.
\end{proposition}
\begin{proof}
Let $s = \operatorname{diag}(1,\zeta,\zeta^2, \dots, \zeta^{n-1})$. Then $sv_i = v_{i+1}$ (with the indices interpreted modulo $n$) and $s$ commutes with $t$. Thus letting $c_m(t)$ be the coefficient of $v_m$ in $tv_{0}$, I get
\[
    tv_i = ts^iv_0 = s^itv_0 = s^i\sum_{m=0}^{n-1} c_m(t)v_{m} = \sum_{m=0}^{n-1} c_m(t)v_{i+m}
\]
where $i+m$ is interpreted modulo $n$.
\end{proof}

\begin{proof}[Proof of Lemma~\ref{lem:rot_fixed_in_quotient}]
This proof will be in two steps: 
\begin{enumerate}
\item Every such $\rho$-fixed point in $(\OPtd/T)(\Bbbk)$ lifts to a $\rho$-fixed point in $\OPtd(\Bbbk)$
\item the $T$-orbit of a $\rho$-fixed point $V_I \subseteq \OPtd(\Bbbk)$, intersected with the $\rho$-fixed points, consists exactly of the $\rho$-fixed points achieved by rotation of $I$.
\end{enumerate}
(1, Lifting fixed points). Suppose $[W]\in (\OPtd/T)(\Bbbk)$ satisfies $\rho\cdot[W]=[W]$ for $W$ a $k$-dimensional linear subspace of $\Bbbk^n$. Then there exists $t\in T(\Bbbk)$ with
\[
\rho\cdot W = t\cdot W.
\]
Define the map
\[
\varphi: T\longrightarrow T\qquad s\mapsto \rho(s) s^{-1} \pmod{\text{scalars}},
\]
where $\rho(s) := \rho s\rho^{-1}$ corresponds to the cyclic rotation of the diagonal entries of an element in $T$. Since the $n$-by-$n$ matrix $(\delta_{i,j} - \delta_{i,j-1 \bmod n})$ has rank $n-1$ with columns orthogonal to the all ones vector, I can choose $s\in T(\overline{\Bbbk})$ with 
\[
\rho(s) s^{-1} \equiv t^{-1}\pmod{\text{scalars}}.
\]
Then
\[
\rho\cdot (s\cdot W) = \rho (s)\cdot (\rho\cdot W) = (\rho s) \cdot (t \cdot W) = \big((\rho s)s^{-1} t\big)\cdot (s\cdot W),
\]
and by our choice of $s$ the diagonal $(\rho s)s^{-1} t$ is scalar; hence $\rho\cdot(s\cdot W)=s\cdot W$ as a point of the Grassmannian. Thus every $\rho$-fixed point of the quotient lifts to a $\rho$-fixed point of $\OPtd(\overline{\Bbbk})$. Finally by Proposition~\ref{prop:rot_fixed_points}, $\OPtd(\overline{\Bbbk})^\rho = \OPtd(\Bbbk)^\rho$.

(2, $T$-orbit intersect with $\rho$-fixed points).
By corollary \ref{cor:rot_fixed} 
the $\rho$-fixed points in $\OPtd$ are of the form $V_I := \bigoplus_{i\in I} L_i$ for $I$ a subset of $\{0,1,\dots,n-1\}$ of size $k$.
Now suppose that $t\cdot V_I = V_J$ for some $t \in T(\Bbbk)$. Choose some $m$ such that $c_m(t) \neq 0$ where $c_m(t)$ are as in Proposition \ref{prop:T_action_on_rot_fixed}. By Proposition \ref{prop:T_action_on_rot_fixed}, for each $i \in I$, the $v_{i+m}$ coefficient of $tv_i$ in the $\rho$-eigenbasis $\{v_0,v_1,v_2, \dots, v_{n-1}\}$ of $V$ is nonzero. Since $tv_i \in tV_I = V_J$, I have $i+m \in J$. Since $I$ and $J$ are both size $k$, this implies $I+m = J$. Conversely, if $t$ is one of the following $n$ matrices:
\[
     t_m := \operatorname{diag}(1,\zeta^m,\zeta^{2m}, \dots, \zeta^{(n-1)m}) \textnormal{ for $0 \leq m < n$}.
\]
Then $t_mV_I = V_{I+m}$.
\end{proof}
From the above counts, the Lefschetz trace theorem, and Theorem~\ref{thm:cohom-dim} I can now show that rotation acts as the identity on $H^i_{c}((X_{k,n}^\circ)_{\Bbbk},\mathbb{Q}_{\ell})$:
\begin{proof}[Proof of Theorem~\ref{thm:main_thm_1.1}]
By the Lefschetz fixed-point formula for finite-order automorphisms \ref{thm:lef-finiteorder},
\[
\sum_i (-1)^i \operatorname{Tr}\big(\rho, H^i_{c}((X_{k,n}^\circ)_{\Bbbk};\mathbb{Q}_{\ell})\big) \;=\; \#(X_{k,n}^\circ)(\Bbbk)^\rho.
\]
Since the characteristic of $\Bbbk$ is coprime to $n$ and $\Bbbk$ is algebraically closed, $\Bbbk$ contains a primitive $n$-th root of unity. From the previous Lemma the right-hand side equals $\frac{1}{n}\binom{n}{k}$. By Theorem~\ref{thm:cohom-dim} the total Betti dimension equals the same number
\[
\sum_i \dim H_c^i((X^\circ_{k,n})_{\Bbbk},\mathbb{Q}_\ell) = \frac{1}{n}\binom{n}{k}
\]
and all the cohomology is even. Since $\rho$ is order $n$ and $\overline{\mathbb{Q}_{\ell}}$ contains the $n$-th roots of unity, $\rho$ acts semisimply with eigenvalues $\zeta^i$. Thus 
\begin{align}\label{eq:sum_roots_of_unity}
\frac{1}{n}\binom{n}{k} =\sum_i (-1)^i \operatorname{Tr}\big(\rho, H^i_{c}((X^\circ_{k,n})_{\Bbbk},\mathbb{Q}_\ell)\big) = \sum_i m_i\zeta^i
\end{align}
where $m_i$ is the multiplicity of the $\zeta^i$-eigenspace in $\bigoplus_i H_c^i((X^\circ_{k,n})_{\Bbbk},\mathbb{Q}_\ell)$. From the total dimension of $H_c^*((X^\circ_{k,n})_{\Bbbk},\mathbb{Q}_\ell)$ equaling the sum of the multiplicities of the eigenspaces, I also get
\begin{align}\label{eq:sum_mult_of_eigenval}
\sum m_i = \frac{1}{n}\binom{n}{k}.
\end{align}
After choosing any embedding $\overline{\mathbb{Q}_{\ell}} \to \mathbb{C}$, using \eqref{eq:sum_mult_of_eigenval} the triangle inequality implies
\[
|\sum_i m_i\zeta^i|_{\mathbb{C}} \leq \sum_i m_i|\zeta^i|_{\mathbb{C}} = \frac{1}{n}\binom{n}{k}
\]
with equality if all $m_i\zeta^i$ have the same argument. Thus \eqref{eq:sum_roots_of_unity} implies that all the eigenspaces $m_i$ are zero except for one, which must have real positive argument. I conclude $m_0 = \frac{1}{n}\binom{n}{k}$ and $m_i = 0$ for $i \neq 0$.

Since $\rho$ is of finite order, it acts semisimply on $H_c^*((X^\circ_{k,n})_{\Bbbk},\mathbb{Q}_\ell)$. Since all of $\rho$'s eigenvalues are $1$ and $\rho$ acts semisimply, it acts trivially on $H_c^*((X^\circ_{k,n})_{\Bbbk},\mathbb{Q}_\ell)$.
\end{proof}

\begin{remark}
It would be interesting to investigate whether analogues of Theorem~\ref{thm:main_thm_1.1} continue to hold for suitable automorphisms of the top-dimensional projected Richardson variety in a maximal partial flag variety in other Lie types.
\end{remark}

\subsection{Proof of Theorem 1.2}\label{subsec:main_cor}

Using the discussion at the beginning of section~\ref{sec:main}, I now translate the cohomological triviality of the rotation into an equality of point-counts for two different $\F_q$-rational forms that become isomorphic over $\F_{q^n}$.

\begin{corollary}\label{cor:eq_Gal_modules}
    Let $k$ and $n$ be coprime and let $q$ be coprime to $n$. When $\ell \nmid n$, $H_c^*((X_{k,n}^\circ)_{\overline{\mathbb{F}_q}},\mathbb{Q}_{\ell})$ and $H_c^*((X_{k,n}^{\circ'})_{\overline{\mathbb{F}_q}},\mathbb{Q}_{\ell})$ are isomorphic as $\mathbb{Q}_{\ell}[\Fr_q]$-modules.
\end{corollary}
\begin{proof}
Let $m_g$ be the isomorphism between $(X_{k,n}^\circ)_{\mathbb{F}_{q^n}}$ and $(X_{k,n}^{\circ'})_{\mathbb{F}_{q^n}}$. Then from the equality
\[
    \Fr \circ g = g \circ \rho \circ \Fr,
\]
I get on the compactly supported \'etale cohomology that
\[
\left(\Fr^{X_{k,n}^{\circ'}}\right)^* = m_{g^{-1}}^* \circ \left(\Fr^{X_{k,n}^\circ}\right)^* \circ \rho^* \circ m^*_g.
\]
By Theorem~\ref{thm:main_thm_1.1}, $\rho$ acts trivially on the cohomology, hence $m_{g}^*$ is the desired isomorphism of $\mathbb{Q}_{\ell}[\Fr]$-modules.
 \end{proof}

\begin{corollary}\label{cor:eq_point_counts}
    Let $k$ and $n$ be coprime and let $q$ be coprime to $n$. I obtain an equality of point counts
\[
\big|\big(X_{k,n}^\circ\big)(\F_q)\big|
\;=\;
\big|\big(X_{k,n}^{\circ'}\big)(\F_q)\big|
\]
\end{corollary}

\begin{proof}
By the Grothendieck--Lefschetz trace formula,
\[
|X(\F_q)|=\sum_i (-1)^i\operatorname{Tr}\big(\mathrm{Frob}_q \:|\: H^i_c(X\otimes\overline{\F_q},\overline{\mathbb{Q}}_\ell)\big).
\]
This combined with corollary \ref{cor:eq_Gal_modules} implies $\big|\big(X_{k,n}^\circ\big)(\F_q)\big|\;=\;\big|\big(X_{k,n}^{\circ'}\big)(\F_q)\big|$
\end{proof}

There is one more technical tool needed to prove the second equality of Theorem~\ref{thm:main_thm_1.2}:

\begin{lemma}\label{lem:moore_det_consequence}
Let $M$ be a $k$-by-$k$ matrix of the form 
\[
M_{i,j} = v_i^{q^{j-1}}, \qquad v = (v_1,v_2, \dots, v_k) \in \mathbb{F}_{q^n}^k.
\]
Then $\mathrm{det}(M) = 0$ if and only if $v_1,v_2, \dots, v_k$ are linearly dependent over $\mathbb{F}_q$.
\end{lemma}
\begin{proof}
The determinant of such a matrix is a Moore determinant \cite{Moore1896} which has a product formula
\[
\det(M) = \prod_{1 \leq i \leq k}\ \prod_{c_1,c_{2}, \dots, c_{i-1} \in \mathbb{F}_q} \left(c_1v_1 + c_2v_2 +\dots+c_{i-1}v_{i-1} + v_i\right)
\]
This vanishes if and only if one of the terms in the product is zero, which happens if and only if there is an $\mathbb{F}_q$ relation between $v_1,v_2, \dots, v_k$.
\end{proof}

\begin{remark}
The same statement holds for general extensions $L/K$ with cyclic Galois group. See \cite[Lemma~3.3]{GXY18} for a proof that applies in the general case.
\end{remark}
\begin{proof}[Proof of Theorem~\ref{thm:main_thm_1.2}]
Applying corollary \ref{cor:cor_of_LS} to corollary \ref{cor:eq_point_counts}, I get
\begin{align*}
    \frac{|\OPtd(\mathbb{F}_q)|}{|T(\mathbb{F}_q)|} = \frac{|\OPtdtwist(\mathbb{F}_q)|}{|T'(\mathbb{F}_q)|}
\end{align*}
Thus it is enough to show that $|\OPtdtwist(\F_q)| = |\Gr(k,n)(\mathbb{F}_q)|$. 

Let $\Gr'(k,n)$ be the $\F_q$-rational form of $\Gr(k,n)_{F_{q^n}}$ corresponding to the $\Fr$-twisted conjugacy class of $\rho$ in $H^1(\Gal(\F_{q^n}/\F_q),\Aut(\Gr(k,n)_{F_{q^n}}))$ (see subsection~\ref{subsec:rat_forms} for more details). Then by the last statement of Proposition~\ref{prop:forms-H1}, $\OPtdtwist$ descends to an open subset of $\Gr'(k,n)$. Since $\rho = g^{-1}\Fr(g)$, the $\Fr$-twisted conjugacy class of $\rho$ includes the identity and hence is trivial in $H^1(\Gal(\F_{q^n}/\F_q),
Aut(\Gr(k,n)_{F_{q^n}}))$. Thus by Proposition~\ref{prop:forms-H1}, $\Gr'(k,n) \cong \Gr(k,n)$ over $\F_q$. Explicitly in the ordered basis consisting of columns of $g$, the Frobenius action on $\Gr'(k,n)(\F_{q^n})$ is
\[
(v_1,v_2,\dots,v_n)\longmapsto (v_n^q,v_1^q,\dots,v_{n-1}^q),
\qquad v_j\in \F_{q^n}^k.
\]
To make the fixed-point condition explicit, represent a point $V \in \Gr'(k,n)$ as the row span of a full-rank
$k\times n$ matrix
\[
M_V=(v_1,\dots,v_n),\qquad v_j\in \overline{\F}_q^{\,k},
\]
modulo left multiplication by $\GL_k(\overline{\F}_q)$.  The Frobenius on $\Gr'(k,n)$ sends
$V$ to
\[
    \mathrm{Fr}_q'(V):=[\mathrm{Fr}_q'(M_V)], \qquad \mathrm{Fr}_q'(M_V):=(v_n^q,v_1^q,\dots,v_{n-1}^q).
\]
Then $V\in \Gr'(k,n)(\F_q)$ if and only if there exists $a \in \GL_k(\overline{\F}_q)$ such that
\begin{align}\label{eq:frob_twist_glk}
M_V=a\,\mathrm{Fr}_q'(M_V).
\end{align}
By Theorem~\ref{thm:LangS} (Lang's theorem), I may choose $h\in \GL_k(\overline{\F}_q)$ with
\[
h^{-1}\Fr_q(h)=a.
\]
Multiplying the relation \eqref{eq:frob_twist_glk} by \(h\) on the left and using \(\Fr_q(h)=ha\), I obtain
\[
hM_V=\mathrm{Fr}_q'(hM_V).
\]
Equivalently, writing \(hM_V=(v_1',\dots,v_n')\), I have
\[
(v_1',\dots,v_n')=( (v_n')^q,(v_1')^q,\dots,(v_{n-1}')^q).
\]
Thus every point $V \in \Gr'(k,n)(\F_q)$ is the row-span of a matrix of the form $(v_1, v_1^q,v_1^{q^2}, \dots, v_1^{q^{n-1}})$ with rank $k$ and $v_1 \in \F_{q^n}^k$. Let $v_1 =(v_{1,1},v_{2,1},v_{3,1}, \dots, v_{k,1})^\top$. \\
I claim that $v_{1,1},v_{2,1},v_{3,1}, \dots, v_{k,1}$ are $\F_q$-linearly independent. To prove this claim, suppose that I have a relation
\[
 \sum_{i=1}^k c_iv_{i,1} = 0, \qquad c_i \in \F_{q}.
\]
Then by applying Frobenius,
\[
    \sum_{i=1}^k c_iv_{i,1}^{q^j} = 0, \qquad \textnormal{ for any } j,
\]
so $(c_1,c_2, \dots, c_k)^\top \in \ker(v_1, v_1^q,v_1^{q^2}, \dots, v_1^{q^{n-1}})^\top$. Since $(v_1, v_1^q,v_1^{q^2}, \dots, v_1^{q^{n-1}})$ has rank $k$, this implies $c_1=c_2 = \dots  = c_k = 0$ proving the claim that $v_{1,1},v_{2,1},v_{3,1}, \dots, v_{k,1}$ are $\F_q$-linearly independent.\\

Then by applying the Frobenius automorphism, $v_{1,1}^{q^j},v_{2,1}^{q^j},\dots, v_{k,1}^{q^j}$ are $\F_q$-linearly independent for any $j$. Finally Lemma~\ref{lem:moore_det_consequence} implies all the cyclic $k$-by-$k$ minors are non-zero. Thus the inclusion $\OPtdtwist(\F_q) \to \Gr'(k,n)(\F_q)$ is an equality of sets.
\end{proof}

\begin{remark}\label{rem:bij_hope}
A natural question is whether Theorem~\ref{thm:main_thm_1.2} can be made bijective. A lot of the approaches that work for $k=1$ use steps that imply $\Pi_{1,n}^{\circ'}/T' \cong \Pi_{1,n}^\circ/T$. 
However, this fails already for $k=2, n=5$. To see that $\Pi_{2,5}^{\circ'}/T' \not\cong \Pi_{2,5}^\circ/T$, note that $\Pi_{2,5}^\circ/T$ is the $C(1,1)$-cluster algebra whose automorphisms are $D_{10}$ (see \cite{BD15}), corresponding to rotation and horizontal reflection of $2$-by-$5$-matrices. 
These are all defined over $\F_q$ hence 
\[
\Aut_{\F_q}(\Pi_{2,5}^\circ/T) = D_{10}.
\] 
However only the automorphisms which commute with rotation descend to automorphisms of $\Pi_{2,5}^{\circ'}/T'$. So  
\[
\Aut_{\F_q}(\Pi_{2,5}^{\circ'}/T') = \Z/5\Z.
\]
\end{remark}

\begin{remark}
    For other Weyl groups, the rational $q$-Catalan numbers are defined in \cite{GLTW24} $q$-ifying a definition of $W$-Catalan numbers given in \cite{Hai94}.  These give the point counts of open Richardson varieties $R_{e,c^k}$ where $c$ is the Coxeter element in the Weyl group. It would be interesting to see if there is an explanation of nice formula involving rational $q$-Catalan numbers in other types similar to Theorem~\ref{thm:main_thm_1.2}.
\end{remark}

\bibliographystyle{alpha}
\bibliography{references}

\end{document}